\documentclass{amsart}

\usepackage{amsmath, amssymb, amsthm, setspace, hyperref, xypic, mathrsfs}
\usepackage{tikz-cd}
\usepackage{xcolor}

\newcommand{\bA}{\mathbb{A}}
\newcommand{\bC}{\mathbb{C}}

\newcommand{\bV}{\mathbb{V}}

\newcommand{\caL}{\mathcal{L}}
\newcommand{\cO}{\mathcal{O}} 
\newcommand{\cV}{\mathcal{V}}
\newcommand{\cX}{\mathcal{X}}

\newcommand*{\sheafhom}{\mathscr{H}\kern -.5pt om}
\newcommand{\funhom}{\underline{\textnormal{Hom}}}
\newcommand{\funisom}{\underline{\textnormal{Isom}}}

\DeclareMathOperator{\Spec}{Spec}
\DeclareMathOperator{\Sym}{Sym}
\DeclareMathOperator{\Aut}{Aut}

\DeclareMathOperator{\Fr}{Fr}

\theoremstyle{definition}
\newtheorem{definition}{Definition}
\numberwithin{definition}{section}

\newtheorem{example}[definition]{Example}

\newtheorem{situation}[definition]{Situation}
\newtheorem{remark}[definition]{Remark}
\theoremstyle{theorem}

\newtheorem{theorem}[definition]{Theorem}

\newtheorem{lemma}[definition]{Lemma}

\title[Embedding Deligne--Mumford stacks into GIT quotient stacks]{Embedding Deligne--Mumford stacks into GIT quotient stacks of linear representations}

\author{Mitchell Faulk}
\address{Mitchell Faulk, Department of Mathematics, 1326 Stevenson Center,  Vanderbilt University, Nashville, TN 37240}
\email{mitchell.m.faulk@vanderbilt.edu}

\author{Chiu-Chu Melissa Liu}
\address{Chiu-Chu Melissa Liu, Department of Mathematics, Columbia University, 2990 Broadway, New York, NY 10027, USA}
\email{ccliu@math.columbia.edu}

\date{}

\begin{document}

\begin{abstract}
We study how to use a suitably ample locally free sheaf of rank $r$ over a proper Deligne--Mumford  stack to furnish an embedding of the stack into a geometric invariant theory  (GIT) quotient stack  $[\bA^\text{ss}/GL_r]$ where $\bA$ is a finite dimensional representation of the general linear group $GL_r$. 


\end{abstract}

\maketitle

\setcounter{tocdepth}{1} 
\tableofcontents

\section{Introduction}

Given a polarized compact complex manifold $(X,L)$, a classical result, which is often called the Kodaira embedding theorem, asserts that for all sufficiently large $m$, the natural evaluation map for sections of $L^m$ defines a closed embedding from $X$ into the projectivization of the dual of the space of global sections.
More generally, if a locally free sheaf $V$ of rank $r$ over $X$ is also specified, then for all sufficiently large $m$, the evaluation map for sections of $V\otimes L^m$ defines a closed embedding from $X$ into the Grassmann variety of $r$-planes in the dual of the space of global sections. 

An extension of the rank one result has been recently obtained for orbifold singularities which are \emph{cyclic}.  Precisely, let $\cX$ denote a compact complex orbifold with cyclic stabilizers.  Let $N$ be the minimal positive integer such that the order of every element in every stabilizer group divides $N$. (The number $N$ is called the order of $\cX$ in \cite{ross2011weighted} and the index of $\cX$ in \cite{abramovich2011stable}.)  Ross and Thomas \cite{ross2011weighted}
define an orbifold polarization on $\cX$ to be an invertible sheaf $\caL$ on $\cX$ such that stabilizer groups acts faithfully on fibers of $\caL$ and that $\caL^N$ descends to an ample line bundle on underlying space $X$. These conditions can also be stated in algebro-geometric language. In fact, the orbifold  $\cX$ is a proper Deligne--Mumford stack
defined over the field $\bC$ of complex numbers, with cyclic geometric stabilizer groups. Let $\pi:\cX\to X$ be the morphism to the coarse moduli space. Then an orbifold polarization $\caL$ consists of a $\pi$-ample invertible sheaf such that $\caL^N =\pi^*M$ where $M$ is an ample invertible sheaf on $X$. Ross and Thomas use global sections of 
$\caL^{m+j}$, where $m\gg 0$ and $0\leq j\leq N$, to construct a closed embedding of the stack $\cX$ into a weighted projective stack. In algebro-geometric language,  
Abramovich and Hassett \cite{abramovich2011stable} prove a similar result for a proper cyclotoic stack $f: \cX \to B$ over a general base scheme $B$ of finite type, equipped with a $\pi$-ample invertible sheaf $\caL$ such that $\caL^N = \pi^*M$ where $M$ is an invertible sheaf on the coarse moduli $\bar{f}: X \to B$ which  is $\bar{f}$-ample. The result described by Ross and Thomas corresponds to the case $B=\Spec \bC$ and $f: \cX \to \Spec \bC$ is smooth. 

The purpose of this note is to describe a natural generalization of the result obtained by Ross and Thomas to the situation where the geometric stabilizers are not necessarily cyclic. We consider a possibly singular Deligne-Mumford stack $\cX$  over a field $k$ which is algebraically closed and of characteristic zero. Let $\pi:\cX\to X$ be the morphism to the coarse moduli. We assume that 
\begin{enumerate}
\item[(i)] $p:\cX \to \Spec k$ is proper
\item[(ii)] $\bar{p}: X\to \Spec k$ is projective, and 
\item[(iii)] $\cX$ possesses a $\pi$-ample
locally free sheaf $\cV$.
\end{enumerate} 
By assumption (iii), all of the geometric stabilizer groups act faithfully on the fibers of $\cV$. At the same time, our assumptions guarantee that there exists a $\pi$-ample locally free sheaf $\cV$  such that   $(\det\cV)^N = \pi^*M$ for some ample invertible sheaf $M$ on $X$, where $N$ is the minimal positive
integer such that the order of each geometric stabilizer group divides $N$.  We use global sections of $(\det \cV)^{mN}\otimes \Sym^i\cV$, for $m\gg 0$ and $j\leq N$, to obtain the following main result. 

\begin{theorem}\label{thm:mainintro}
Under assumptions (i) through (iii), there is an embedding of $\mathcal{X}$ into a 
geometric invariant theory (GIT) quotient  stack 
\begin{align}\label{eqn:A^ss/GL_r}
[\mathbb{A}^{\textnormal{ss}}/GL_r]
\end{align}
where $r$ is the rank of the locally free sheaf $\cV$, $\bA$ is a linear representation of the general linear group $GL_r$, and 
$\mathbb{A}^{\textnormal{ss}}$ is the semistable locus determined by the determinant character. 
\end{theorem}

The assumptions we impose are natural and arguably offer another answer to a problem posed by Kresch in \cite{kresch2009geometry}, namely, to understand when a Deligne--Mumford stack should be called ``(quasi-)projective.'' There, he offers the following equivalences.   

 \begin{theorem}[Kresch]
 Let $\cX$ be a proper Deligne--Mumford stack over a field $k$ of characteristic zero.  Then the following are equivalent:
 \begin{enumerate}
 \item $\cX$ has a projective coarse moduli space and possesses a $\pi$-ample locally free sheaf, where $\pi:\cX \to X$ is
 the morphism to the coarse moduli. 
\item $\cX$ has a projective coarse moduli space and possesses a generating sheaf.
\item $\cX$ admits a closed embedding into a smooth proper Deligne--Mumford stack with projective coarse moduli space. 
 \end{enumerate}
 \end{theorem}

In Section \ref{sec:preliminaries}, we recall useful definitions and results concerning GIT and stacks, including various notions of moduli spaces and stability. 

We identify in Section \ref{sec:TwistedGrassmannians} a certain class of stacks of the form \eqref{eqn:A^ss/GL_r} which we call \emph{twisted Grassmannians} because they generalize both the construction of ordinary Grassmann varieties and the construction of weighted projective spaces.  

Our main result is stated and proved in Sections \ref{sec:embeddingstatement} and \ref{sec:embeddingproof}. As mentioned earlier, the result extends known results of a similar flavor. Indeed  when the geometric stabilizers are cyclic and rank$\, \cV=1$, we obtain an embedding of $\mathcal{X}$ into a weighted projective stack (with coarse moduli that is a weighted projective space). At the same time, 
when $\cX=X$ is a projective scheme and the rank of $\mathcal{V}$ is $r$,  we obtain an embedding of $\cX=X$ into a Grassmann variety of $r$-planes in a suitable vector space, which reduces to a projective space when $r = 1$.

In future work, we plan to generalize our construction to a Deligne--Mumford stack $f: \cX\to B$ over a base scheme $B$ of finite type, assuming  (i) $f:\cX\to B$ is flat and proper, (ii) the coarse moduli $\bar{f}:\cX \to B$ is flat and projective, (iii) $\cX$ possesses a $\pi$-ample locally free sheaf
$\cV$, where $\pi:\cX\to X$ is the morphism to the coarse moduli, such that $(\det \cV)^N = \pi^*M$ for some $\bar{f}$-ample invertible sheaf
on $X$. We seek to find additional hypotheses that we may impose in order to obtain an embedding of $\cX$ into certain twisted Grassmann \emph{bundle} defined over $B$.

\subsection*{Acknowledgments}  We thank Johan de Jong for helpful conversations and Daniel Halpern-Leistner for helpful communications. 
The second author  is partially supported by NSF DMS-1564497.

\section{Preliminaries}\label{sec:preliminaries}

Throughout, we will work over a fixed field $k$ of characteristic zero that we assume to be algebraically closed. This means in particular that any scheme that appears will be understood to be a scheme over $k$ and that the fiber product of schemes will be over $k$, unless otherwise indicated. In addition, for a positive integer $n$, the notation $\mathbb{A}^n$ means affine $n$-space over $k$ and similarly for the general linear group $GL_n$.

\begin{remark}
There are two primary reasons that we assume $k$ has characteristic zero and is algebraically closed. 
\begin{enumerate}
\item While $GL_n$ is linearly reductive in characteristic zero, it is not linearly reductive in positive or mixed characteristics. On a closely related note, if $GL_n$ acts on an affine scheme $\Spec(R)$, then the natural morphism 
\[
[\Spec(R)/GL_n] \longrightarrow \Spec(R^{GL_n})
\]
will only be an adequate moduli space, but not a good moduli space\footnote{The notion of a good moduli space was introduced by Alper \cite{alper2013good}. The precise definition will be given in Definition \ref{def:goodmoduli}.} in general. See \cite[Remark 2.5]{alper2018existence}. 

\item For a proper scheme $X$ over an algebraically closed field $k$, a classical result says that an invertible sheaf is very ample if and only if it separates points and tangent vectors (see, e.g.  \cite[Proposition II.7.3]{hartshorne2013algebraic}). The proof exploits the Nullstellensatz to assert that all residue fields must be equal to $k$. Our intention in this note is to obtain a generalization of this result, though there may even be suitable extensions relative to fields which are not algebraically closed or even relative to more general Noetherian schemes $B$.      
\end{enumerate}
Nevertheless, we conjecture it is possible to obtain results similar to those in this note whilst working relative to any field or, even, relative to an arbitrary base scheme $B$, provided some additional hypotheses are imposed. 
\end{remark}

\subsection{Linearizations}

\begin{example}\label{eg:action}
Any morphism $G \to GL_n$ of group schemes (over $k$) determines an action of $G$ on $\mathbb{A}^n$ through the fundamental (left) action of $GL_n$ on $\mathbb{A}^n$. More generally, for a scheme $X$, if $\mathbb{A}^n_X$ denotes the base change to $X$, then any morphism $\rho : G \to GL_n$ of group schemes determines an action of $G$ on $\mathbb{A}^n_X$ (over $X$) through base change:
\[
\Sigma_\rho : G \times \mathbb{A}_X^n \longrightarrow \mathbb{A}^n_X. 
\]
\end{example}

\begin{definition}[{\cite[II, Definition 1.7.8]{EGA}}]\label{def:vectorbundle}
Let $X$ be a scheme. Let $E$ be a locally free sheaf over $X$. The \emph{vector bundle associated to $E$} is the scheme 
\[
\mathbb{V}(E) = \Spec_X(\Sym(E))
\]
together with the affine morphism $\mathbb{V}(E) \to X$. It is important to note that the functor $\mathbb{V}$ is a \emph{contravariant} functor from the category of locally free sheaves to the category of vector bundles over $X$. (This is because the functor $\Sym$ is covariant and the functor $\Spec_X$ is contravariant.)

As an alternative characterization, consider the functor which associates to any scheme $f : T \to X$ the set of morphisms of $\mathcal{O}_T$-modules 
\[
\sheafhom(f^*E, \mathcal{O}_T).
\]
This functor is representable by the scheme $\mathbb{V}(E)$. 
\end{definition}

\begin{example}
Let $X$ be a scheme. For a positive integer $n$, 
\[
\mathbb{V}(\mathcal{O}_X^n) = \Spec_X(\Sym(\mathcal{O}_X^n)) \simeq \mathbb{A}_X^n
\]
is affine $n$-space over $X$. Indeed for any scheme $f : T \to X$, the $T$-points are 
\[
\mathbb{V}(\mathcal{O}_X^n)(T) = \sheafhom(f^*(\mathcal{O}_X^n), \mathcal{O}_T) \simeq \sheafhom(\mathcal{O}_T^n, \mathcal{O}_T) \simeq \Gamma(T, \mathcal{O}_T)^n = \mathbb{A}_X^n(T). 
\]
\end{example}

\begin{definition}[{\cite[Definition 1.6]{mumford1994geometric}}]
Let $X$ be a scheme. Let $G$ be a group scheme. Let $\sigma : G \times X \to X$ be an action of $G$ on $X$. Let $E$ be a locally free sheaf over $X$. By a \emph{$G$-linearization} is meant an isomorphism of sheaves 
\[
\phi : \sigma^*E \xrightarrow{\;\;\sim \;\;} p_2^*E
\] 
over $G \times X$ satisfying the cocycle relation 
\[
p_{23}^* \phi \circ (1_G \times \sigma)^* \phi = (\mu \times 1_X)^*\phi
\]
over $G \times G \times X$, where $\mu : G \times G \to G$ denotes the group law. 
\end{definition}

\begin{lemma}[{\cite[Section 1.3]{mumford1994geometric}}]
Let $X$ be a scheme. Let $G$ be a group scheme. Let $\sigma : G \times X \to X$ be an action of $G$ on $X$. Let $E$ be a locally free sheaf over $X$. The data of a $G$-linearization of $E$ is equivalent to the data of an action 
\[
\Sigma : G \times \mathbb{V}(E)  \longrightarrow \mathbb{V}(E)
\]
of $G$ on $\mathbb{V}(E)$ such that 
\[
\xymatrix{
	G \times \mathbb{V}(E) \ar[r]^{\;\;\;\;\;\;\Sigma} \ar[d] & \mathbb{V}(E) \ar[d] \\
	G \times X \ar[r]_{\;\;\;\;\sigma} & X
}
\]
commutes and such that $\Sigma$ is a bundle isomorphism of the vector bundles $G \times \mathbb{V}(E)$ over $G \times X$ and $\mathbb{V}(E)$ over $X$. 
\end{lemma}

\begin{example}\label{example:linearizationrho}
Let $X$ be a scheme. Let $G$ be a group scheme. Let $G$ act on $X$  by the trivial action 
\[
\sigma = p_2 : G \times X \to X. 
\]
Any morphism $\rho : G \to GL_{n}$ of group schemes determines a $G$-linearization of $\mathcal{O}_X^n$ that can be described as follows. The vector bundle corresponding to $\mathcal{O}_X^n$ can be identified with affine $n$-space over $X$
\[
\mathbb{V}(\mathcal{O}_X^n) = \mathbb{A}_X^n.
\]
Example \ref{eg:action} describes an action 
\[
\Sigma_\rho : G \times \mathbb{A}_X^n \to \mathbb{A}_X^n
\]
of $G$ on $\mathbb{A}_X^n$ (over $X$). This action determines a $G$-linearization of $\mathcal{O}_X^n$. 
\end{example}

\begin{lemma}[{\cite[Section 1.3]{mumford1994geometric}}]\label{lemma:linearizationtensorproduct}
The tensor product of two $G$-linearized sheaves enjoys the structure of an induced $G$-linearization.  
\end{lemma}

\begin{lemma}\label{lemma:linearizationdirectsum}
The direct sum of two $G$-linearized sheaves enjoys the structure of an induced $G$-linearization. 
\end{lemma}

\subsection{Algebraic spaces, stacks, and inertia}

We follow closely the notation and terminology of \cite{alper2013good}. 

\begin{definition}[{\cite[Section 2]{alper2013good}}]
An \emph{algebraic space} (over $k$) is a sheaf of sets 
\[
X : (\text{Sch}/k)_{\text{\'etale}}^{\text{op}} \longrightarrow \text{Set}
\]
such that 
\begin{enumerate}
\item The diagonal $\Delta_{X/k} : X \to X \times_k X$ is representable by schemes and quasi-compact. 
\item There exists an \'etale, surjective morphism $U \to X$ where $U$ is a scheme (over $k$).  
\end{enumerate}
\end{definition}

\begin{definition}[{\cite[Section 2]{alper2013good}}]
An \emph{Artin stack} (over $k$) is a category 
\[
p : \mathcal{X} \longrightarrow (\text{Sch}/k)_{\text{\'etale}}
\]
over $k$ with the following properties. 
\begin{enumerate}
\item The category $\mathcal{X}$ is a stack in groupoids over $(\text{Sch}/k)_{\text{\'etale}}$. 
\item The diagonal $\Delta_{\mathcal{X}/k} : \mathcal{X} \to \mathcal{X} \times_k \mathcal{X}$ is representable by algebraic spaces, separated, and quasi-compact. 
\item There exists a smooth, surjective map $U \to \mathcal{X}$ where $U$ is a scheme over $k$. 
\end{enumerate}
If the smooth, surjective morphism $U \to \mathcal{X}$ can be taken to be \'etale, then we call $\mathcal{X}$ a \emph{Deligne--Mumford} (DM) stack. 
\end{definition}

\begin{definition}
Let $p : \mathcal{X} \to \Spec k$ be an Artin stack. The \emph{inertia stack} $\mathcal{I}\mathcal{X}$ is the fibre product 
\[
\xymatrix{
\mathcal{I}\mathcal{X} \ar[rr] \ar[d] &  & \mathcal{X} \ar[d]^{\Delta_{\mathcal{X}/k}} \\
\mathcal{X} \ar[rr]_{\!\!\!\!\!\Delta_{\mathcal{X}/k}}& &  \mathcal{X} \times_k \mathcal{X}
}.
\]
\end{definition}

\begin{definition}[{\cite[Section 2.1]{alper2013good}}]
Let $p : \mathcal{X} \to \Spec k$ be an Artin stack. For a geometric point $x : \Spec K \to \mathcal{X}$ of $\mathcal{X}$, let $G_x$ denote the fibre product 
\[
\Aut x = \Spec K \times_{\mathcal{X}} \mathcal{I}\mathcal{X}.
\]
Because $\Delta_{\mathcal{X}/k}$ is separated, $G_x$ is a group scheme (see Lemma 99.19.1 of \cite{stacksproject}). We call $G_x$ the \emph{stabilizer group} associated to $x$. 
\end{definition}

\subsection{Moduli}

In this section, we consider the following situation. 

\begin{situation}\label{sit:pi}
Let $\mathcal{X}$ be an Artin stack. Let $X$ be an algebraic space. Let $\pi : \mathcal{X} \to X$ be a morphism of stacks (over $\Spec k$). 
\end{situation}

\begin{definition}[{\cite[Section 1]{conrad2005keel}}]
In Situation \ref{sit:pi}, the morphism $\pi$ is called a \emph{coarse moduli space} if the following properties are satisfied. 
\begin{enumerate}
\item The morphism $\pi$ is initial among all morphisms over $\Spec k$ to algebraic spaces over $\Spec k$. 
\item The induced map 
\[
[\mathcal{X}(k)] \longrightarrow X(k)
\]
is bijective, where $[\mathcal{X}(k)]$ denotes the set of isomorphism classes of objects in the small category $\mathcal{X}(k)$. (Remember that $k$ is algebraically closed.)
\end{enumerate}
\end{definition}

Here is a well-known situation in which a coarse moduli space exists. 

\begin{situation}\label{sit:coarse}
Let $p : \mathcal{X} \to \Spec k$ be an Artin stack. Assume the following hypotheses. 
\begin{enumerate}
\item The morphism $p$ is separated and locally of finite presentation. 
\item The projection $\mathcal{I}\mathcal{X} \to \mathcal{X}$ is finite.
\end{enumerate}
\end{situation}

\begin{theorem}[\cite{conrad2005keel,keel1997quotients}]
In Situation \ref{sit:coarse}, there is a coarse moduli space $\pi : \mathcal{X} \to X$. Moreover, it satisfies the following additional properties. 
\begin{enumerate}
\item The structure morphism $\bar{p} : X \to \Spec k$ is separated. 
\item The morphism $\pi$ is proper and quasi-finite. 
\end{enumerate}
\end{theorem}

There are other types of moduli, such as good and tame moduli,  studied by Alper \cite{alper2013good}. 

\begin{definition}[{\cite[Definition 4.1]{alper2013good}}]\label{def:goodmoduli}
In Situation \ref{sit:pi}, the morphism $\pi$ is called a \emph{good moduli space} if the following properties are satisfied. 
\begin{enumerate}
\item The morphism $\pi$ is cohomologically affine.
\item The natural morphism $\mathcal{O}_X \to \pi_*\mathcal{O}_{\mathcal{X}}$ is an isomorphism. 
\end{enumerate}
\end{definition}

\begin{definition}[{\cite[Definition 7.1]{alper2013good}}]
In Situation \ref{sit:pi}, the morphism $\pi$ is called a \emph{tame moduli space} if the following properties are satisfied. 
\begin{enumerate}
\item The morphism $\pi$ is a good moduli space. 
\item The induced map 
\[
[\mathcal{X}(k)] \longrightarrow X(k)
\]
is a bijection of sets.  
\end{enumerate}
\end{definition}

By imposing a suitable tameness condition, one can ensure that a coarse moduli space is always tame. 

\begin{definition}[{\cite[Definition 3.1]{abramovich2008tame}}]
In Situation \ref{sit:coarse}, we say that $\mathcal{X}$ is \emph{tame} if the pushforward functor 
\[
\pi_* : \text{QCoh} \mathcal{X} \longrightarrow \text{QCoh} X
\]
is exact, where $\pi : \mathcal{X} \to X$ denotes the coarse moduli space. 
\end{definition}

\begin{lemma}[{\cite[Theorem 3.2]{abramovich2008tame}}]
In Situation \ref{sit:pi}, the following are equivalent. 
\begin{enumerate}
\item The stack $\mathcal{X}$ is tame. 
\item For each geometric point $x : \Spec K \to \mathcal{X}$, the group scheme $G_x$ is linearly reductive. 
\end{enumerate}
\end{lemma}

\begin{lemma}[{\cite[Remark 7.3]{alper2013good}}]
In Situation \ref{sit:coarse}, if $\mathcal{X}$ is tame, then the coarse moduli space $\pi : \mathcal{X} \to X$ is a tame moduli space. 
\end{lemma}

\begin{lemma}[{\cite[Lemma 2.8]{nironi2008moduli} or \cite{olsson2003quot}}]\label{lem:projection}
In Situation \ref{sit:coarse}, write $\pi : \mathcal{X} \to X$ for the coarse moduli space. Let $\mathcal{F}$ be a quasi-coherent sheaf over $\mathcal{X}$ and $G$ a coherent sheaf over $X$. Assume $\mathcal{X}$ is DM and tame. Then there is a natural isomorphism  
\[
\pi_*(\pi^*G \otimes \mathcal{F}) \simeq G \otimes \pi_*\mathcal{F}. 
\]
\end{lemma}

\begin{lemma}[{\cite[Remark 1.4]{nironi2008moduli}}]\label{lem:H-bullet}
In Situation \ref{sit:coarse}, if $\mathcal{X}$ is tame, then for each quasicoherent sheaf $\mathcal{F}$ over $\mathcal{X}$, there are isomorphsims 
\[
H^\bullet(\mathcal{X}, \mathcal{F}) \simeq H^\bullet(X, \pi_*\mathcal{F}).
\] 
\end{lemma}

\subsection{Stability}

\begin{definition}[{\cite[Definition 11.1]{alper2013good}}]\label{def:stable}
Let $p : \mathcal{X} \to \Spec k$ be a quasi-compact Artin stack. Let $\mathcal{L}$ be an invertible sheaf of $\mathcal{O}_{\mathcal{X}}$-modules. Let $x : \Spec K \to \mathcal{X}$ be a geometric point of $\mathcal{X}$.
\begin{enumerate}
\item We say $x$ is \emph{semi-stable} with respect to $\mathcal{L}$ if there is a section $\sigma \in \Gamma(\mathcal{X}, \mathcal{L}^n)$ for some $n > 0$ such that $\sigma(x) \ne 0$ and $\mathcal{X}_{\sigma} \to 
 \Spec k$ is cohomologically affine. 
\item We say $x$ is \emph{stable} with respect to $\mathcal{L}$ if there is a section $\sigma \in \Gamma(\mathcal{X}, \mathcal{L}^n)$ for some $n > 0$ such that $\sigma(x) \ne 0$, $\mathcal{X}_{\sigma} \to \Spec k$ is cohomologically affine, and $\mathcal{X}_\sigma$ has closed orbits.  
\end{enumerate}
Denote by $\mathcal{X}_{\mathcal{L}}^{\text{ss}}$ and $\mathcal{X}_{\mathcal{L}}^{\text{s}}$ the corresponding open substacks corresponding to the semistable and stable points respectively. 
\end{definition}

\begin{example}\label{example:quotient}
Let $p : Y \to \Spec k$ be a quasi-compact algebraic scheme. Let $G$ be a linearly reductive group scheme over $k$ acting on $Y$. Let $L$ be a $G$-linearization on $Y$. Then $L$ determines an invertible sheaf $\mathcal{L}$ on the quotient stack $\mathcal{X} = [Y/G]$. As a result, we obtain open substacks $\mathcal{X}_{\mathcal{L}}^{\text{ss}}$ and $\mathcal{X}_{\mathcal{L}}^{\text{s}}$. Moreover, via the natural morphism $Y \to \mathcal{X}$, these determine open subschemes $Y_{\text{ss}}^L$ and $Y_{\text{s}}^L$ of $Y$. According to \cite[Subsection 13.5]{alper2013good}, the open subschemes $Y_{\text{ss}}^L$ and $Y_{\text{s}}^L$ are the same as those defined in \cite{mumford1994geometric}. 
\end{example}

The next result follows from  \cite[Theorem 13.6]{alper2013good} or equivalently \cite[Theorem 1.10]{mumford1994geometric}.

\begin{lemma}\label{lem:alperquotient}
Let $\mathcal{X}$ be the quotient stack of Example \ref{example:quotient} with the invertible sheaf $\mathcal{L}$. 
\begin{enumerate}
\item There is a good moduli space $\pi : \mathcal{X}_{\mathcal{L}}^{\textnormal{ss}} \to X$ with $X$ an open subscheme of $\textnormal{Proj}_k \bigoplus_{d \geqslant 0} (p_* \mathcal{L}^d)^G$. 
\item There is an open subscheme $V \subset X$ such that $\pi^{-1}(V) = \mathcal{X}_{\mathcal{L}}^{\textnormal{s}}$ and $\pi_{\mathcal{X}_{\mathcal{L}}^{\textnormal{s}}} : \mathcal{X}_{\mathcal{L}^{\textnormal{s}}} \to V$ is a tame moduli space. 
\item If $\mathcal{X}_{\mathcal{L}}^{\textnormal{ss}}$ is quasi-compact, then there exists an ample line bundle $M$ on $X$ such that $\pi^*M \simeq \mathcal{L}^N$ for some $N$. 
\end{enumerate}
\end{lemma}

\subsection{Frame bundles}

\begin{definition}[{\cite[Situation 08JT]{stacksproject}}]
Let $p : \mathcal{X} \to \Spec k$ be an Artin stack. Let $\mathcal{E}, \mathcal{F}$ be quasi-coherent sheaves of $\mathcal{O}_{\mathcal{X}}$-modules. Let $\funhom_{\mathcal{X}}(\mathcal{E}, \mathcal{F})$ denote the functor 
\[
\funhom_{\mathcal{X}}(\mathcal{E}, \mathcal{F}) : (\text{Sch}/k)^{\text{op}} \longrightarrow \text{Set}
\]
defined as follows. If $T$ is a scheme, then an element of $\funhom_{\mathcal{X}}(\mathcal{E}, \mathcal{F})(T)$ consists of a pair $(h,u)$, where $h$ is a morphism $h : T \to \mathcal{X}$ of stacks and $u : h^*\mathcal{E} \to h^*\mathcal{F}$ is an $\mathcal{O}_T$-module morphism. The functor $\funhom_{\mathcal{X}}(\mathcal{E}, \mathcal{F})$ is equipped with a natural morphism to $\mathcal{X}$ described by sending a pair $(h,u)$ to the morphism $h$.  
\end{definition}

\begin{example}\label{eg:homvectorbundle}
Let $p : X \to \Spec k$ be a scheme. Let $E,F$ be quasi-coherent sheaves of $\mathcal{O}_X$-modules. The functor $\funhom_{X}(E,F)$ is representable by the scheme $\mathbb{V}(\sheafhom(E, F))$.
\end{example}

\begin{example}
Given a quasi-coherent sheaf $\mathcal{E}$ over $\mathcal{X}$, the functor $\funhom_{\mathcal{X}}(\mathcal{E}, \mathcal{O}_{\mathcal{X}}^r)$ is naturally a $GL_r$-torsor, where the action is left multiplication. 
\end{example}

\begin{lemma}\label{lem:Hom-i-m}
Let $p : \mathcal{X} \to \Spec k$ be an Artin stack. Let $\mathcal{V}$ be a locally free sheaf of rank $r$ over $\mathcal{X}$. Set $\mathcal{L} = \det \mathcal{V}$. For each pair of positive integers $i,m$, there is a natural morphism 
\[
\funhom_{\mathcal{X}}(\mathcal{V}, \mathcal{O}_{\mathcal{X}}^r) \longrightarrow \funhom_k(p_*(\Sym^i\mathcal{V} \otimes \mathcal{L}^m), \Sym^i (\mathcal{O}_S^r))
\]
over $\Spec k$. Moreover, the morphism is equivariant if the right-hand side is viewed as a $GL_r$-torsor where the action is described by $ \Sym^i\rho \otimes (\det \rho)^{\otimes m} $ if $\rho$ denotes the fundamental representation of $GL_r$. 
\end{lemma}

\begin{proof}
For notational convenience, write 
\[
\psi_{i,k} : p^*p_*(\Sym^i \mathcal{V} \otimes \mathcal{L}^m) \longrightarrow \Sym^i \mathcal{V} \otimes \mathcal{L}^m
\]
for the natural adjunction morphism. Let $T$ be a scheme over $S$. Let $h : T \to \mathcal{X}$ be a morphism of stacks. Let $u : h^*\mathcal{V} \to \mathcal{O}_T^r$ be morphism of $\mathcal{O}_T$-modules. Note that for each $i$, the morphism $u$ determines a morphism 
\[
h^*(\Sym^i \mathcal{V}) \simeq  \Sym^i (h^*\mathcal{V})  \xrightarrow{\;\; u \;\;} \Sym^i(\mathcal{O}_T^r).
\]
In addition, $u$ determines a morphism 
\[
h^*\mathcal{L} = h^*\det \mathcal{V} \simeq \det h^*\mathcal{V} \xrightarrow{\;\;u\;\;} \det(\mathcal{O}_T^r) \simeq \mathcal{O}_T
\]
and so for each $m$ a morphism  
\[
h^*\mathcal{L}^m \xrightarrow{\;\;u\;\;} \mathcal{O}_T
\]
as well. Using this, we see that we have 
\begin{itemize}
\item a morphism $p \circ h : T \to \Spec k$ depicted by 
\[
T \xrightarrow{\;\; h\;\;}\mathcal{X} \xrightarrow{\;\; p \;\;} \Spec k
\]
\item and an $\mathcal{O}_T$-module morphism 
\[
(p \circ h)^*p_*(\Sym^i \mathcal{V} \otimes \mathcal{L}^m) \longrightarrow (p \circ h)^*(\Sym^i(\mathcal{O}_{\Spec k}^r))
\]
depicted by 
\[
h^*p^*p_*(\Sym^i\mathcal{V} \otimes \mathcal{L}^m) \xrightarrow{u \circ h^*\psi_{i,k}\;}  \Sym^i (\mathcal{O}_T^r) \simeq (p \circ h)^* (\Sym^i(\mathcal{O}_{\Spec k}^r)).
\]
\end{itemize}
Some details omitted. 
\end{proof}

\begin{definition}
Let $p : \mathcal{X} \to \Spec k$ be an Artin stack. Let $\mathcal{E}, \mathcal{F}$ be quasi-coherent sheaves of $\mathcal{O}_{\mathcal{X}}$-modules. We can consider the subfunctor 
\[
\funisom_{\mathcal{X}}(\mathcal{E}, \mathcal{F}) \subset \funhom_{\mathcal{X}}(\mathcal{E}, \mathcal{F})
\]
whose value on a scheme $T$ consists of the set of pairs $(h,u)$ where $h$ is a morphism $h : T \to \mathcal{X}$ of stacks and $u :  h^*\mathcal{E} \to h^*\mathcal{F}$ is an $\mathcal{O}_T$-module \emph{isomorphism}. 
\end{definition}

\begin{definition}
Let $p : \mathcal{X} \to \Spec k$ be an Artin stack. Let $\mathcal{V}$ be a locally free sheaf over $\mathcal{X}$ of rank $r$. The \emph{frame bundle of $\mathcal{V}$} is defined to be the functor 
\[
\Fr(\mathcal{V}) = \funisom_{\mathcal{X}}(\mathcal{V}, \mathcal{O}_{\mathcal{X}}^r)
\]
over $\mathcal{X}$. 
\end{definition}

\begin{lemma}\label{lem:frame-i-m}
Let $p : \mathcal{X} \to \Spec k$ be an Artin stack. Let $\mathcal{V}$ be a locally free sheaf of rank $r$ over $\mathcal{X}$. Set $\mathcal{L} = \det \mathcal{V}$. For each pair of positive integers $i,m$, there is a natural morphism 
\[
\Fr(\mathcal{V})  \longrightarrow   \funhom_k(p_*(\Sym^i\mathcal{V} \otimes \mathcal{L}^m), \Sym^i(\mathcal{O}_{\Spec k}^r)).
\]
over $\Spec k$. Moreover, the morphism is equivariant if the right-hand side is viewed as a $GL_r$-torsor where the action is described by $ \Sym^i\rho \otimes (\det \rho)^{\otimes m} $ if $\rho$ denotes the fundamental representation of $GL_r$. 
\end{lemma}

\begin{proof}
Follows immediately from Lemma \ref{lem:Hom-i-m}.
\end{proof}

\subsection{Ample and generating sheaves}

In this section, we consider the following situation. 

\begin{situation}\label{sit:locallyfree}
In Situation \ref{sit:coarse}, write $\pi : \mathcal{X} \to X$ for the coarse moduli space. Let $\mathcal{E}$ denote a locally free sheaf over $\mathcal{X}$. Make the following assumptions. 
\begin{enumerate}
\item The stack $\mathcal{X}$ is DM. 
\item There is a positive integer $N$ such that $N$ is a multiple of the order of each geometric stabilizer group $G_x$ for $x : \Spec k \to \mathcal{X}$. Moreover, it is sufficient to select $N$ to be minimal with this property. 
\end{enumerate}
\end{situation}

\begin{definition}
In Situation \ref{sit:locallyfree}, we say that $\mathcal{E}$ is \emph{generating} if for each quasi-coherent sheaf $\mathcal{F}$ over $\mathcal{X}$, the natural morphism 
\[
\pi^*\pi_* \sheafhom(\mathcal{E}, \mathcal{F}) \otimes \mathcal{E} \longrightarrow \mathcal{F}
\]
is surjective. 
\end{definition}

\begin{lemma}[{\cite[Section 5.2]{kresch2009geometry}}]
In Situation \ref{sit:locallyfree}, the sheaf $\mathcal{E}$ is generating if and only if for each geometric point of $\mathcal{X}$, with $G$ the geometric stabilizer group, the representation of $G$ on the fibre of $\mathcal{E}$ at the geometric point contains every irreducible representation of $G$. 
\end{lemma}

\begin{definition}[{\cite[Definition 2.1]{nironi2008moduli}}]
In Situation \ref{sit:locallyfree}, we say that $\mathcal{E}$ is \emph{$\pi$-ample} if for each geometric point of $\mathcal{X}$, with $G$ the geometric stabilizer group, the representation of $G$ on the fibre of $\mathcal{E}$ at the geometric point is faithful.
\end{definition}

The next result follows from a theorem of Burnside, which enjoys several extensions, such as one due to Steinberg \cite[Theorem 7]{steinberg2014burnside}, which states that every irreducible representation of a finite group is contained in some symmetric power of a faithful representation. 

\begin{lemma}\label{lem:generating}
In Situation \ref{sit:locallyfree}, if $\mathcal{E}$ is $\pi$-ample, then $\bigoplus_{i=0}^N \Sym^i\mathcal{E}$ is a generating sheaf. 
\end{lemma}

The next result can be found in the proof of  \cite[Lemma 2.12]{edidin2001brauer}. 

\begin{lemma}\label{lem:framerepresentable}
In Situation \ref{sit:locallyfree}, if $\mathcal{E}$ is $\pi$-ample, then the frame bundle $\Fr(\mathcal{E})$ is representable by an algebraic space. 
\end{lemma}

\begin{lemma}
In Situation \ref{sit:locallyfree}, assume $\mathcal{E}$ has rank $r$ and is $\pi$-ample. Then the natural morphism $\Fr(\mathcal{E}) \to \mathcal{X}$ descends to an isomorphism 
\[
[\Fr(\mathcal{E})/GL_{r}] \xrightarrow{\;\;\sim\;\;} \mathcal{X}
\]
of stacks over $S$. 
\end{lemma}

\subsection{Amplitude and properness}

In this section, we consider the following situation. 

\begin{situation}\label{sit:proper}
In Situation \ref{sit:coarse}, write $\pi : \mathcal{X} \to X$ for the coarse moduli space. Assume the following.  
\begin{enumerate}
\item The stack $\mathcal{X}$ is DM.
\item The morphism $p : \mathcal{X} \to \Spec k$ is  
proper.  (In particular, it is of finite type and of finite presentation.) 
\end{enumerate}
\end{situation}

\begin{lemma}
In Situation \ref{sit:proper}, the following are true. 
\begin{enumerate}
\item The stack $\mathcal{X}$ is tame. 
\item There is a positive integer $N$ such that $N$ is a multiple of the order of each geometric stabilizer group $G_x$ for $x : \Spec k \to \mathcal{X}$. Moreover, it is sufficient to select $N$ to be minimal with this property. 
\end{enumerate}
\end{lemma}

\begin{proof}
Statement (1) is because the characteristic of $k$ is zero. Statement (2) follows because the morphism $\mathcal{IX} \to \mathcal{X}$ is finite and $\mathcal{X} \to \Spec k$ is proper. 
\end{proof}

\begin{lemma}\label{lem:proper}
In Situation \ref{sit:proper}, the following are true. 
\begin{enumerate}
\item The morphism $\bar{p}: X\to \Spec k$ is proper.
\item The stack $\mathcal{X}$ is Noetherian. 
\item The space $X$ is Noetherian. 
\item If $\mathcal{L}$ is an invertible sheaf on $\mathcal{X}$, then there is an invertible sheaf $L$ on $X$ and an isomorphism 
\[
\mathcal{L}^N \simeq \pi^*L
\]
of sheaves over $\mathcal{X}$.
\item The pushforward functor $\pi_*$ maps coherent sheaves to coherent sheaves. 
\item The pullback functor $\pi^*$ maps coherent sheaves to coherent sheaves. 
\end{enumerate}
\end{lemma}

\begin{proof}
Statement (1) follows because $\pi: \cX \to X$ is a universal homeomorphism. 
Statement (2) follows because $p$ is of finite type over $\Spec k$.  
Statement (3) follows because $\bar{p}$ is of finite type over $\Spec k$.
A proof of Statement (4) is found in \cite[Lemma 2.3.7]{abramovich2011stable} in the case where all of the stabilizer groups are cyclic, but the same argument extends mutatis mutandis to our general setting. Statement (5) follows because $\pi$ is proper and $X$ is Noetherian (see 
\cite[III, Theorem 3.2.1]{EGA}). Statement (6) follows because $\mathcal{X}$ is Noetherian (see  e.g. \cite[Theorem 16.3.7 (3)]{vakil2017rising}). 
\end{proof}

Recall the following result concerning amplitude with respect to \emph{proper} morphisms (over a Noetherian base). 

\begin{lemma}[Serre's vanishing criterion {\cite[III, Proposiiton 2.6.1]{EGA}}]\label{lem:Serrescheme}
Let $f : X \to \Spec k$ a proper morphism of schemes. Let $L$ be an invertible sheaf over $X$. 
The following conditions are equivalent. 
\begin{enumerate}
\item The sheaf $L$ is $f$-ample. 
\item For each coherent $\mathcal{O}_X$-module $F$, there is an integer $m_0$ such that 
\[
H^q(X, F \otimes L^m) = 0 \hspace{10mm} \text{for $q > 0$ and $m \geqslant m_0$}.
\]
\end{enumerate} 
\end{lemma}

Here is an extension to Situation \ref{sit:proper}. 

\begin{lemma}\label{lem:Serrealgebraicspace}
In Situation \ref{sit:proper}, let $L$ be an invertible sheaf over $X$. The following are equivalent. 
\begin{enumerate}
\item The sheaf $L$ is $\bar{p}$-ample. 
\item The algebraic space $X$ is a scheme and for each coherent sheaf $F$ over $X$, there is an integer $m_0$ such that 
\[
H^q(X, F \otimes L^m) = 0 \hspace{10mm} \text{for $q > 0$ and $m \geqslant m_0$}. 
\]
\end{enumerate}
\end{lemma}

\begin{proof}
The implication (2) implies (1) is immediate from Lemma \ref{lem:Serrescheme}. Assume (1). Then $X$ is a scheme by 
\cite[Lemma 0D32]{stacksproject}. The rest of (2) then follows from Lemma \ref{lem:Serrescheme} again. 
\end{proof}


\begin{lemma}\label{lem:locallyfree}
In Situation \ref{sit:proper}, let $L$ be an invertible sheaf over $X$. Assume $L$ is $\bar{p}$-ample. For each locally free sheaf $\mathcal{E}$ over $\mathcal{X}$, there is an integer $m_0$ satisfying 
\[
H^q(X,\pi_*\mathcal{E} \otimes L^m) = 0 \hspace{10mm} \text{for each $q > 0$ and $m \geqslant m_0$.}
\]
\end{lemma}

\begin{proof}
The sheaf $\mathcal{E}$ is coherent over $\mathcal{X}$, 
and so the sheaf $\pi_*\mathcal{E}$ is coherent over $X$ 
by Lemma \ref{lem:proper}. It follows that the sheaves $\pi_*\mathcal{E} \otimes L^m$ are coherent over $X$. 
The statement follows from Lemma \ref{lem:Serrealgebraicspace}. 
\end{proof}

\section{Twisted Grassmannians}\label{sec:TwistedGrassmannians}

Let $r,n$ be positive integers. Let $\mu = (\mu_0, \ldots, \mu_n)$ be a sequence of nonnegative integer multiplicities with at least one multiplicity strictly positive. Consider the functor which associates to any scheme $T$ the additive group  
\begin{align}\label{eqn:A(r,mu)(T)}
\bigoplus_{i=0}^n \text{Hom}_{k}(\Gamma(T,\mathcal{O}_T)^{\mu_i}, \Sym^i(\Gamma(T,\mathcal{O}_T)^r)).
\end{align}
This is representable by an affine scheme which we denote by $\mathbb{A}(r, \mu)$. Moreover, the affine scheme $\mathbb{A}(r,\mu)$ is an additive group scheme over $k$. 

The affine group scheme $GL_{r}$ naturally acts on $\mathbb{A}(r,\mu)$ over $k$ via a morphism of schemes 
\[
\sigma(r,\mu) : GL_r \times \mathbb{A}(r,\mu) \longrightarrow \mathbb{A}(r,\mu)
\]
satsfiying the property that for each scheme $T$, the corresponding map of sets 
\[
GL_r(T) \times \mathbb{A}(r,\mu)(T) \longrightarrow \mathbb{A}(r,\mu)(T) 
\]
is induced by viewing each 
\[
\text{Hom}_{k}(\Gamma(T,\mathcal{O}_T)^{\mu_i}, \Sym^i(\Gamma(T,\mathcal{O}_T)^r)) \simeq ( \Gamma(T, \mathcal{O}_T)^{\mu_i})^\vee \otimes_{k} \Sym^i(\Gamma(T,\mathcal{O}_T)^r )
\]
as a $GL_r(T)$-set via the trivial action of $GL_r(T)$ on $\Gamma(T, \mathcal{O}_T)^{\mu_i}$ and the action of $GL_r(T)$ on $\Sym^i(\Gamma(T,\mathcal{O}_T)^r)$ induced by the fundamental action of $GL_r(T)$ on $\Gamma(T,\mathcal{O}_T)^r$.

In addition, for each nonnegative integer $m$, there is an action 
\[
\sigma(r,\mu, m) : GL_r \times \mathbb{A}(r,\mu) \longrightarrow \mathbb{A}(r,\mu)
\] 
which is obtained from $\sigma(r,\mu)$ by twisting by $(\det)^{\otimes m}$. 

Let $\mathcal{A}(r,\mu,m)$ denote the resulting quotient stack 
\begin{align}\label{eqn:A-r-mu}
\mathcal{A}(r,\mu,m) = [\mathbb{A}(r,\mu)/GL_r]
\end{align}
over $k$.  

The determinant morphism $\det : GL_r \to \mathbb{G}_m$ determines a $GL_r$-linearization of the structure sheaf $\mathcal{O}_{\mathbb{A}(r,\mu)}$ via Example \ref{example:linearizationrho}. Let $\mathcal{L}_0$ denote the corresponding invertible sheaf over $\mathcal{A}(r,\mu,m)$. Denote by $\mathcal{A}(r,\mu,m)_{\mathcal{L}_0}^{\text{ss}}$ the open substack corresponding to the locus of semistable geometric points as in Definition \ref{def:stable}. 

\begin{definition}\label{def:cG(r,mu)}
Define the \emph{twisted Grassmann stack} associated to the data $(r,\mu,m)$ to be the open substack 
\[
\mathcal{G}(r,\mu,m) = \mathcal{A}(r,\mu,m)_{\mathcal{L}_0}^{\text{ss}}
\]
of $\mathcal{A}(r,\mu,m)$ over $k$. 
\end{definition}

\begin{lemma}
The twisted Grassmann stack $p : \mathcal{G}(r,\mu,m) \to \Spec k$ admits a good moduli space 
\[
\pi : \mathcal{G}(r,\mu,m) \longrightarrow G(r,\mu,m)
\]
where $G(r,\mu,m)$ is an open subscheme of $\textnormal{Proj} \bigoplus_{d \geqslant 0} H^0(\mathcal{G}(r,\mu,m), \mathcal{L}^d)^{GL_r}$. 
\end{lemma}

\begin{proof}
See Lemma \ref{lem:alperquotient} (a). 
\end{proof}

\begin{example}\label{eg:weightedprojectivespace}
For example, if $r = 1$, then $\mathcal{G}(1,\mu,m)$ is the weighted projective stack 
\[
\mathcal{P}(\overbrace{m, \ldots,m}^{\mu_0}, \overbrace{m+1, \ldots, m+1}^{\mu_1}, \ldots, \overbrace{m+n, \ldots, m+n}^{\mu_n})
\]
as in \cite{abramovich2011stable}. The good moduli space $G(1,\mu,m)$ is the weighted projective space 
\[
G(1,\mu,m) = \mathbb{P}(\overbrace{m, \ldots, m}^{\mu_0}, \overbrace{m+1, \ldots, m+1}^{\mu_1}, \ldots, \overbrace{m+n, \ldots, m+n}^{\mu_n}).
\]
\end{example}

\begin{example}
As another example, suppose $n = 1$, $\mu_0 = 0$, and $r < \mu_1$.  Then the affine scheme $\mathbb{A}(r,\mu)$ is the one whose $k$-valued points can be identified with the vector space $\text{Hom}_k(k^{\mu_1}, k^r)$. For $m = 0$, at the level of $k$-points, the action of $GL_r$ is just post-composition, and the semistable locus consists of those linear maps with full rank. It follows that the quotient stack $\mathcal{G}(r, \mu, 0)$ can be identified with the Grassmann variety of $r$-planes in $\mathbb{A}^{\mu_1}$ (see e.g. \cite[Example 2.14]{hoskins2014stratifications}). 
\end{example}

\section{Embedding statement}\label{sec:embeddingstatement}

In this section, we consider the following situation. 

\begin{situation}\label{sit:embed}
In Situation \ref{sit:proper}, let $\mathcal{V}$ be a locally free sheaf of rank $r$ over $\mathcal{X}$. Make the following assumptions. 
\begin{enumerate}
\item The sheaf $\mathcal{V}$ is $\pi$-ample. 
\item There is a $\bar{p}$-ample invertible sheaf $L$ over $X$ such that $(\det \mathcal{V})^N \simeq \pi^*L$, where $N$ is the smallest positive integer which is a multiple of the order of each stabilizer group $G_x$ for $x : \Spec k \to \mathcal{X}$. 
\end{enumerate} 
\end{situation}
Then, because $\mathcal{V}$ is $\pi$-ample, the frame bundle $\Fr(\mathcal{V})$ is representable by an algebraic space (Lemma \ref{lem:framerepresentable}). In fact, because $L$ pulls back to an ample sheaf over $\Fr(\mathcal{V})$, the frame bundle is representable by a scheme (according to  \cite[Lemma 69.14.12]{stacksproject}). Call this scheme $Y$. It enjoys a natural morphism 
\[
q : Y \longrightarrow \mathcal{X}
\]
to $\mathcal{X}$.

For each suitably large positive integer $m$, there is a morphism of schemes 
\begin{align}\label{eqn:f-m}
f_m : Y \longrightarrow \mathbb{A}(r,\mu(m)),
\end{align}
where $\mathbb{A}(r,\mu(m))$ is a certain affine scheme as desecribed as in \eqref{eqn:A(r,mu)(T)}. 
Indeed, the sheaf 
\[
\bigoplus_{i=0}^N \Sym^i \mathcal{V}
\] 
is generating (see Lemma \ref{lem:generating}).  For each $0 \leqslant i \leqslant N$, define the sheaf
\[
V_i(m) := p_*(\cX, \Sym^i \mathcal{V} \otimes \pi^*L^{m}) =  \bar{p}_* \pi_*(\Sym^i\mathcal{V} \otimes L^m) 
\]
over $\Spec k$, where the second equality follows from Lemma \ref{lem:projection}. The sheaf $V_i(m)$ corresponds
to a free $k$-module of rank $\mu(m)$, or equivalently a $\mu(m)$-dimensional vector space over $k$, for some
non-negative integer $\mu(m)$. 
There is an isomorphism of affine spaces over $k$ 
\[
\bV \left( \bigoplus_{i=0}^N \sheafhom(V_i(m), \Sym^{i}\cO^r_{\Spec k})\right) \simeq \mathbb{A}(r, \mu(m))
\]
where $\mu(m) = (\mu_0(m), \ldots, \mu_N(m))$ and where $\mathbb{A}(r, \mu(m))$ is defined as in \eqref{eqn:A(r,mu)(T)}. Then using Lemma \ref{lem:frame-i-m}, we obtain the morphism \eqref{eqn:f-m}.  

According to Lemma \ref{lem:Hom-i-m}, the morphism is $GL_r$-equivariant, and thus it descends to a morphism of stacks 
\begin{align}\label{eqn:F_m}
F_m : \mathcal{X} \longrightarrow \mathcal{A}(r,\mu(m))
\end{align}
over $\Spec k$, where $\mathcal{A}(r,\mu(m)) = [\mathbb{A}(r,\mu(m))/GL_{r}]$ as in \eqref{eqn:A-r-mu}. In summary, we have, for each $m$, a Cartesian square of stacks 
\[
\xymatrix{
	Y \ar[rr]^{\!\!\!\!\!\!\!\!\!\!\!\!\!\!\!\!\!\!f_m} \ar[d]_q & & \mathbb{A}(r, \mu(m)) \ar[d] \\
	\mathcal{X} \ar[rr]^{\!\!\!\!\!\!\!\!\!\!\!\!\!\!\!\!\!\!F_m}& & \mathcal{A}(r,\mu(m))
}
\] 
over $\Spec k$. 

The next section will be devoted to proving that the morphism \eqref{eqn:f-m} is an embedding of schemes for $m$ sufficiently large, as stated in the following lemma. 

\begin{lemma}\label{lemma:embedding}
In Situation \ref{sit:embed}, the morphism \eqref{eqn:f-m} of schemes is an embedding for $m$ sufficiently large. 
\end{lemma}

Assuming Lemma \ref{lemma:embedding} for now, we obtain the following main result. 

\begin{theorem}\label{theorem:embed}
In Situation \ref{sit:embed}, for $m$ sufficiently large, the morphism \eqref{eqn:F_m} of stacks factors through the semistable locus and thus determines an embedding of stacks 
\[
\mathcal{X} \longrightarrow \mathcal{G}(r, \mu(m)).
\]
\end{theorem}

\begin{proof}
According to Example \ref{example:quotient}, it suffices to show that the morphism \eqref{eqn:f-m} factors through the semistable locus.  Because there is a $\bar{p}$-ample invertible sheaf $L$ over $X$ such that $(\det \mathcal{V})^N \simeq \pi^*L$ we can guarantee that, for $m$ suitably large, the pushforward 
\[
V_0(m) = p_*((\det \mathcal{V})^{mN})
\]
is nontrivial, say of rank $\ell$. Decompose the morphism $f_m$ into 
\[
f_m = (f_{m,0}, \ldots, f_{m,N})
\]
where 
\[
f_{m,i} : Y \longrightarrow \mathbb{V}(\sheafhom(V_i(m), \Sym^i\mathcal{O}_{\Spec k})).
\]
In particular, we may identify $f_{m,0}$ with 
\[
f_{m,0} : Y \longrightarrow \mathbb{A}^{\ell}.
\] 
The group scheme $GL_r$ acts on $\mathbb{A}^{\ell}$ via multiplication by $(\det \rho)^{mN}$, where $\rho$ denotes the fundamental representation of $GL_r$. If $z_1, \ldots, z_{\ell}$ denote coordinates on $\mathbb{A}^{\ell}$, then each describes a $GL_r$-equivariant section of the equivariant line bundle described by the structure sheaf together with the $GL_r$-linearization determined by the determinant character.  

Let $y$ be a point of $Y$. Let $x$ denote the image of $y$ under the quotient morphism $q : Y \to \mathcal{X}$. There is a section $s$ of $(\det \mathcal{V})^{mN}$ such that $s(x) \ne 0$. As a result, the image $f_{m,0}(y)$ is nonzero. There is a coordinate $z_j$ such that $z(f_{m,0}(y)) \ne 0$. The subset $\mathbb{A}(r,\mu(m))_{z_j}$ is affine. Hence $f_m(y)$ is semistable in $\mathbb{A}(r,\mu(m))$. 
\end{proof}

\section{Embedding proof}\label{sec:embeddingproof}

This section is devoted to proving Lemma \ref{lemma:embedding}. Throughout, we assume we are in Situation \ref{sit:embed}. 

Let us set up some convenient notation. Let $x : \Spec k \to \mathcal{X}$ be a closed point of $\mathcal{X}$ with stabilizer group $G_x$. Then $\mathcal{V}_x$ is a faithful representation of $G_x$, and, as a result, $G_x$ acts on the \emph{right} on the $k$-module
\[
\text{Hom}_k(\mathcal{V}_x, k^r).
\]
On the other hand, this $k$-module also enjoys a \emph{left} action by the group scheme $GL_r$ described by post-composition. 

At the same time, the point $x$ determines a Cartesian square  of stacks
\[
\xymatrix{ \text{Isom}_k(\mathcal{V}_x, k^r)/G_x \ar[r] \ar[d] & Y\ar[d] \\
BG_x \ar[r] & \mathcal{X}}.
\]
Let $\mathcal{J}_x$ denote the ideal sheaf of $\mathcal{O}_{\mathcal{X}}$ determined by $BG_x$. For any locally free sheaf $\mathcal{E}$ over $\mathcal{X}$, there is a long exact sequence in cohomology, a portion of which may be identified with 
\[
H^0(\mathcal{X}, \mathcal{E}) \longrightarrow H^0(BG_x, \mathcal{E}_x) \longrightarrow H^1(\mathcal{X}, \mathcal{J}_x \otimes \mathcal{E}). 
\] 

\begin{lemma}
There is a positive integer $m_0$ such that for each geometric point $x : \Spec k \to \mathcal{X}$, we have  
\[
H^1(\mathcal{X}, \mathcal{J}_x \otimes (\det \mathcal{V})^{mN} \otimes \Sym^i \mathcal{V}) = 0 \hspace{10mm} 0 \leqslant i \leqslant N \; \text{and} \; m \geqslant m_0.
\]
\end{lemma}

\begin{proof}
For each geometric point $x$, Lemma \ref{lem:locallyfree} guarantees the existence of a positive integer $m_0(x)$ satisfying 
\[
H^1(\mathcal{X}, \mathcal{J}_x \otimes (\det \mathcal{V}))^{mN} \otimes \Sym^i \mathcal{V}) = 0 \hspace{10mm} 0 \leqslant i \leqslant N \; \text{and} \; m \geqslant m_0(x).
\]
The fact that $\mathcal{X}$ is proper implies that we may choose $m_0$ uniformly. 
\end{proof}

For a closed point $x : \Spec k \to \mathcal{X}$, define the affine scheme
\[
\mathbb{A}_x(m) = \bigoplus_{i=0}^{N} \mathbb{A}_x^i(m)
\]
where $\mathbb{A}_x^i(m)$ is the affine scheme whose $k$-points can be identified with 
\[
\text{Hom}_k([(\det \mathcal{V}_x)^{mN} \otimes \Sym^i\mathcal{V}_x]^{G_x}, \Sym^i(k^r)).
\]
The group scheme $GL_r$ acts on $\mathbb{A}_x^i(m)$ on the \emph{left} by the action described by $(\det \rho)^{mN} \otimes \rho^i$ where $\rho$ is the fundamental representation. In particular, for $i = 0$, we have an identification of $k$-modules
\[
[(\det \mathcal{V}_x)^{mN} \otimes \Sym^0\mathcal{V}_x]^{G_x} \simeq k,
\]
and so the affine scheme $\mathbb{A}_x^0(m)$ is identified with $1$-dimensional affine space $\mathbb{A}^1$. On this piece, the action of $GL_r$ corresponds to multiplication by $(\det \rho)^{mN}$.

\begin{lemma}
There is a positive integer $m_0$ such that for each geometric point $x : \Spec k \to \mathcal{X}$, the natural morphisms 
\[
\mathbb{A}_x(m) \longrightarrow \mathbb{A}(r, \mu(m)) \hspace{10mm} m \geqslant m_0
\]
are injective. 
\end{lemma}

\begin{proof}
There is a positive integer $m_0$ such that the natural maps 
\[
H^0(\mathcal{X}, (\det \mathcal{V}))^{mN} \otimes \Sym^i\mathcal{V}) \longrightarrow H^0(BG_x,(\det \mathcal{V}))^{mN} \otimes \Sym^i\mathcal{V})
\]
are surjective for $0 \leqslant i \leqslant N$ and $m \geqslant m_0$. At the level of $k$-points, we have an identification 
\[
\mathbb{A}_x^i(m) = \text{Hom}_k(H^0(BG_x, (\det \mathcal{V})^{mN} \otimes \Sym^i \mathcal{V})), \Sym^i(k^r))
\]
and so the lemma follows. 
\end{proof}

As a result, for each geometric point $x : \Spec k \to \mathcal{X}$, the restriction of the morphism \eqref{eqn:f-m} to the fiber over $x$ fits into a commutative diagram 
\[
\xymatrix{
	\text{Isom}_k(\mathcal{V}_x, k^r)/G_x \ar[r]^{\;\;\;\;\;\;\;\;\;f_{m,x}} \ar[d] & \mathbb{A}_x(m) \ar[d] \\
	Y \ar[r]^{\!\!\!\!\!\!\!\!\!f_m} & \mathbb{A}(r, \mu(m))
}
\] 
where the vertical arrows are closed immersions for $m$ sufficiently large. 

\begin{lemma}\label{lem:fiberimmersion}
There is a positive integer $m_0$ such that for each geometric point $x : \Spec k \to \mathcal{X}$, the morphisms 
\[
f_{m,x} : \textnormal{Isom}_k(\mathcal{V}_x, k^r)/G_x \longrightarrow \mathbb{A}_x(m) \hspace{10mm} m \geqslant m_0
\]
are locally closed immersions. 
\end{lemma}

\begin{proof}
The affine scheme $\mathcal{V}_x^\vee$ can be identified with $\Spec k[p_1, \ldots, p_r]$. In this way, we have an identification of $k$-algebras 
\[
\Sym^\bullet \mathcal{V}_x \simeq k[p_1, \ldots, p_r].
\] 
The action of $G_x$ on $\mathcal{V}_x$ corresponds to one on $\Spec k[p_1, \ldots, p_r]$. The invariant subring is finitely generated, so there are a finite number of polynomials 
\[
t_\ell = t_\ell(p_1, \ldots, p_r) \hspace{15mm} 1 \leqslant \ell \leqslant R 
\]
such that 
\[
k[p_1, \ldots, p_r]^{G_x} \simeq k[t_1, \ldots, t_R]. 
\]
Moreover, we may arrange that each $t_\ell$ is homogeneous of degree $d_\ell$ and the degrees are arranged in ascending order $d_1 \leqslant \cdots \leqslant d_R$ with the highest degree being at most the order of $G_x$ (see \cite{noether1915endlichkeitssatz}), which is a number that divides $N$. 

As a result, each $t_\ell$ can be considered as an element of $(\Sym^{d_\ell}\mathcal{V}_x)^{G_x}$. Moreover, for each $d$, a basis for $(\Sym^{d}\mathcal{V}_x)^{G_x}$ may be identified with a subset of monomials of the form 
\[
t^{\alpha} := t_1^{\alpha_1} \cdots t_R^{\alpha_R}
\]
with $\alpha_\ell \geqslant 0$ satisfying 
\begin{align}\label{eqn:alpha}
d_1\alpha_1 + \cdots + d_R \alpha_R = d.
\end{align}
(The basis may not consist of the full set of monomials satisfying \eqref{eqn:alpha} because there could be relations among the $t_\ell$s.) In addition, we may ensure that the basis for $(\Sym^{d_\ell}\mathcal{V}_x)^{G_x}$ includes the monomial $t_\ell$.  

The action of $G_x$ on $\text{Hom}_k(\mathcal{V}_x, k^r)$ is compatible with the decomposition 
\[
\text{Hom}_k(\mathcal{V}_x, k^r) \simeq \bigoplus_{j=1}^r \mathcal{V}_x^\vee
\]
of $k$-modules. 
For each $j = 1, \ldots, r$, identify the $j$th summand with 
\[
\Spec k[p_{1j}, \ldots, p_{rj}]
\]
so that the $p_{ij}$ describe global coordinates on the affine scheme $\text{Hom}_k(\mathcal{V}_x, k^r)$. For each $j$, the group $G_x$ acts on the affine scheme $\Spec k[p_{1j}, \ldots, p_{rj}]$. If we set 
\[
t_{\ell j} = t_\ell(p_{1j}, \ldots, p_{rj}),
\]
then we have an identification 
\[
k[p_{ij}]^{G_x} = k[t_{\ell j}].
\]

The subscheme $\text{Isom}_k(\mathcal{V}_x, k^r)$ can be identified with the affine scheme 
\[
\Spec k[p_{ij}, 1/\det(p_{ij})].
\] 
There is a smallest positive integer $a$ such that $(\det p_{ij})^{a}$ is invariant under $G_x$. Moreover, the number $a$ divides $N$. There is a polynomial $D = D(t_{\ell j})$ such that $D = (\det p_{ij})^{a}$.  We have an identification of affine schemes 
\[
\text{Isom}_k(\mathcal{V}_x, k^r)/G_x \simeq \Spec k[t_{\ell j}, 1/D]. 
\]
There is a positive integer $b$ such that $ab = N$ so that $D^{bm} = (\det p_{ij})^{mN}$. Notice that $D^{bm}$ can be identified with an element of the $k$-module $(\det \mathcal{V}_x)^{mN}$.

Let $e_1, \ldots, e_r$ denote a \emph{dual} basis for $k^r$. A dual basis for $\Sym^d(k^r)$ can then be described by 
\[
e^\beta := e_1^{\beta_1} \cdots e_r^{\beta_r}
\]
where $\beta_j \geqslant 0$ satisfy 
\begin{align}\label{eqn:beta}
\beta_1 + \cdots + \beta_r = d.
\end{align}

Consider the dual space of $\mathbb{A}_x(m)$. This space can be identified with the $k$-module
\[
\bigoplus_{d=0}^N (\det \mathcal{V}_x)^{mN} \otimes \text{Hom}_k(\Sym^d(k^r), (\Sym^d \mathcal{V}_x)^{G_x}).
\]
With our notation, the $d$th piece admits a basis described by 
\[
D^{bm} \otimes e^{\beta} \otimes t^\alpha
\]
where \eqref{eqn:alpha} and \eqref{eqn:beta} hold. In particular, the zeroth piece is one-dimensional spanned by the element 
\[
z^0:= D^{bm} \otimes e^{0} \otimes t^{0}.
\]
Let $U_0$ denote the open subset of $\mathbb{A}_x(m)$ where $z^0$ does not vanish. According to the proof of Theorem \ref{theorem:embed}, the morphism 
\[
f_{m,x} : \text{Isom}_k(\mathcal{V}_x,k^r)/G_x \longrightarrow \mathbb{A}_x(m)
\] 
factors through $U_0$. 

If we restrict our attention to this subscheme, where $z^0$ is invertible, then we may obtain each $t_{\ell j}$ as a linear combination of 
\[
f^*_{m,x}\left(\frac{D^{bm} \otimes e^{\beta(j, \ell)} \otimes t^{\alpha(\ell)}}{z^0}\right)
\]
where 
\[
\beta(j, \ell) = (0, \ldots,0, \overbrace{d_\ell}^{j}, 0, \ldots, 0)
\]
and 
\[ 
\alpha(\ell) = (0, \ldots, 0 ,\overbrace{1}^{\ell}, 0, \ldots, 0).
\]
The result follows. 
\end{proof}

\begin{lemma}\label{lem:injective1}
Given distinct closed points $y,y' : \Spec k \to Y$ of $Y$, there is a positive integer $m_0(y,y')$ such that for each $m \geqslant m_0(y,y')$ the morphism 
\[
f_m : Y \longrightarrow \mathbb{A}(r, \mu(m))
\]
satisfies $f_m(y) \ne f_m(y')$. In addition, there are open neighborhoods $U$ and $U'$ of $q(y)$ and $q(y')$ respectively in $\mathcal{X}$, such that for each pair of distinct points $z \in q^{-1}(U)$ and $z' \in q^{-1}(U')$ and for each $m \geqslant m_0(y,y')$, we have $f_m(z) \neq f_m(z')$. 
\end{lemma}

\begin{proof}
If $y,y'$ lie in the same $GL_r$-orbit, then they map to the same point of $\mathcal{X}$, and we may choose $m_0(y,y')$ from Lemma \ref{lem:fiberimmersion}. Suppose $y,y'$ do not lie in the same orbit. Then they determine distinct points $x,x'$ of $\mathcal{X}$. Let $\mathcal{J}_{x,x'}$ denote the ideal sheaf of $\mathcal{O}_{\mathcal{X}}$ corresponding to their union. For each $i = 0, \ldots, N$, there is a positive integer $m_0(y,y', i)$ such that 
\[
H^1(\mathcal{X}, \mathcal{J}_{y,y'} \otimes \mathcal{L}^{mN} \otimes \Sym^i\mathcal{V}) = 0 \hspace{10mm} m \geqslant m_0(y,y',i).
\]
Let $m_0(y,y')$ denote the maximum of these integers. Then the map 
\[
H^0(\mathcal{X}, (\det \mathcal{V}))^{mN} \otimes \Sym^i\mathcal{V}) \longrightarrow (\Sym^i \mathcal{V}_x)^{G_x} \oplus (\Sym^i \mathcal{V}_{x'})^{G_{x'}}
\]
is surjective for each $i = 0, \ldots, N$. Because $\mathcal{V}$ is $\pi$-ample, there is an $i$ such that 
\[
(\Sym^i \mathcal{V}_x)^{G_x} \ne 0.
\]
It follows that there is a section $s$ of $\mathcal{L}^{mN} \otimes \Sym^i\mathcal{V}$ such that $s(x) \ne 0$ but $s(x') = 0$. This means that $f_m(y) \ne f_m(y')$ for each $m \geqslant m_0(y,y')$. 

In fact, notice that the choice of $m(y,y')$ depends only on the images $x,x'$ of $y,y'$ in $\mathcal{X}$. The other statement of the lemma then follows because $f_m$ is continuous.  
\end{proof}

\begin{lemma}\label{lem:injectiveonclosedpoints}
For sufficiently large $m$, the morphism 
\[
f_m : Y \longrightarrow \mathbb{A}(r, \mu(m))
\] 
is injective on closed points.  
\end{lemma}

\begin{proof}
Using Lemma \ref{lem:injective1}, the properness of $\mathcal{X}$ allows us to cover $\mathcal{X} \times \mathcal{X}$ by a finite number of open neighborhoods $U_1, \ldots, U_\ell$ together with a finite number of positive integers $m_{0,1}, \ldots, m_{0,\ell}$ such that the corresponding morphisms $f_1, \ldots, f_\ell$ determined by these integers are injective on $U_1, \ldots, U_\ell$ respectively. Note that if the morphism $f$ determined by $m$ separates $x$ and $y$, then any morphism determined by $m' \geqslant m$ will also separate $x$ and $y$. As a result, we may take $m_0$ to be the maximum of $m_{0,1}, \ldots, m_{0,\ell}$. 
\end{proof}

\begin{lemma}
Given a closed point $y : \Spec k \to Y$ of $Y$, there is a positive integer $m_0(y)$ such that for each $m \geqslant m_0(y)$, the morphism of stalks
\[
f_{m,y}^\sharp : \mathcal{O}_{\mathbb{A}(r, \mu(m)), f_m(y)} \longrightarrow f_{m,*} \mathcal{O}_{Y,y}
\]
is surjective. In addition, there is a neighborhood $U$ of $q(y)$ in $\mathcal{X}$ such that for each $z \in q^{-1}(U)$ and each $m \geqslant m_0(y)$, the morphism $f_{m,z}^\sharp$ is surjective. 
\end{lemma}

\begin{proof}
Let $\mathfrak{m}_y$ denote the maximal ideal in $\mathcal{O}_{Y,y}$. Let $x$ denote the image of $y$ in $\mathcal{X}$, and let $\mathcal{J}_x$ denote the corresponding ideal sheaf of $\mathcal{O}_{\mathcal{X}}$. For each locally free sheaf $\mathcal{E}$ over $\mathcal{X}$, there is an exact sequence in cohomology, a portion of which is 
\[
H^0(\mathcal{X}, \mathcal{J}_x \otimes \mathcal{E}) \longrightarrow (\mathfrak{m}_y/\mathfrak{m}_y^2 \otimes \mathcal{E}_x)^{G_x} \longrightarrow H^1(\mathcal{X}, \mathcal{J}_x^2 \otimes \mathcal{E}).
\]
Using $\mathcal{E} = (\det \mathcal{V})^{mN} \otimes \Sym^i\mathcal{V}$, Lemma \ref{lem:locallyfree} gives, for each $i$, a positive integer $m_0(y,i)$ such that  
\[
H^1(\mathcal{X}, \mathcal{J}_x^2 \otimes (\det \mathcal{V})^{mN} \otimes \Sym^i\mathcal{V}) = 0 \hspace{10mm} m \geqslant m_0(y,i). 
\]
Let $m_0(y)$ denote the maximum of the $m_0(y,i)$. Because $G_x$ is linearly reductive, we may decompose $\mathfrak{m}_y/\mathfrak{m}_y^2$ into irreducible representations 
\[
\mathfrak{m}_y/\mathfrak{m}_y^2 = \bigoplus_{\ell} W_\ell.
\] 
Using that $\mathcal{V}$ is $\pi$-ample, for each $\ell$, there is a positive integer $i_\ell$ such that 
\[
(W_\ell \otimes (\det \mathcal{V}_x)^{mN} \otimes \Sym^{i_\ell}\mathcal{V}_x)^{G_x} \ne 0. 
\]
As a result, upon writing $a = f_m(y)$, the morphism of $k$-modules 
\[
\mathfrak{m}_a/\mathfrak{m}_a^2 \longrightarrow \mathfrak{m}_y/\mathfrak{m}_y^2
\]
is surjective. It follows that the morphism of local rings 
\[
\mathcal{O}_{\mathbb{A}(r, \mu(m)), a} \longrightarrow f_{m, *}\mathcal{O}_{Y,y}
\]
is surjective too (see e.g. \cite[Lemma 0E8M]{stacksproject}).  

The other statement of the lemma follows from the fact that $f_{m,y}^*$ being surjective represents an open condition on $\mathcal{X}$.  
\end{proof}

\begin{lemma}
For sufficiently large $m$, the morphism 
\[
f_m^\sharp : \mathcal{O}_{\mathbb{A}(r, \mu(m))} \longrightarrow f_{m,*} \mathcal{O}_{Y}
\]
is surjective. 
\end{lemma}

\begin{proof}
The properness of $\mathcal{X}$ allows us to cover $\mathcal{X}$ by a finite number of open neighborhoods $U_1, \ldots, U_\ell$ together with a finite number of positive integers $m_{0,1}, \ldots, m_{0,\ell}$ such that the corresponding morphisms $f_1, \ldots, f_\ell$ satisfy the property that $f_{m_j,y}^\sharp$ is surjective for each $y \in U_j$. Note that if the morphism $f_{m,y}^\sharp$ is surjective, then $f_{m',y}^\sharp$ is surjective for each $m' \geqslant m$. As a result, we may take $m_0$ to be the maximum of $m_{0,1}, \ldots, m_{0,\ell}$. 
\end{proof}

\bibliographystyle{siam} 
\bibliography{amplesheavesbib}

\end{document}